\documentclass[10pt,a4paper]{article}
\usepackage{latexsym,amssymb,amsmath,amsthm,amsfonts,amscd}
\usepackage{color,enumerate}
\usepackage[pdftex,final]{graphicx}
\usepackage{flafter}

\newtheorem{lemma}{Lemma}[section]
\newtheorem{proposition}[lemma]{Proposition}
\newtheorem{theorem}[lemma]{Theorem}

\newtheorem{definition}[lemma]{Definition}
\newtheorem{example}[lemma]{Example}

\newcommand{\ud}{\mathrm{d}}
\newcommand{\RR}{\mathbb{R}}
\newcommand{\f}{\frac}

\newcommand{\xx}{|x|^2}
\newcommand{\yy}{|y|^2}
\newcommand{\xy}{\langle x,y\rangle}

\newcommand{\pp}[2]{\frac{\partial{#1}}{\partial{#2}}}

\newcommand{\pppp}[4]%
  {\frac{\partial^3{#1}}{\partial{#2}\partial{#3}\partial{#4}}}

\newcommand{\p}{\phi}
\newcommand{\ps}{\phi(s)}
\newcommand{\pab}{\alpha\phi(\frac{\beta}{\alpha})}

\newcommand{\gab}{\alpha\phi(b^2,\frac{\beta}{\alpha})}

\newcommand{\sq}{\frac{(\alpha+\beta)^2}{\alpha}}

\renewcommand{\a}{\alpha}
\renewcommand{\b}{\beta}
\newcommand{\ab}{(\alpha,\beta)}
\newcommand{\ta}{\tilde\alpha}
\newcommand{\tb}{\tilde\beta}
\newcommand{\ha}{\hat\alpha}
\newcommand{\hb}{\hat\beta}
\newcommand{\ba}{\bar\alpha}
\newcommand{\bb}{\bar\beta}

\newcommand{\aij}{a_{ij}}

\newcommand{\bi}{b_i}
\newcommand{\bj}{b_j}

\newcommand{\hbi}{\hat b_i}

\newcommand{\bbi}{\bar b_i}
\newcommand{\bbj}{\bar b_j}

\newcommand{\bij}{b_{i|j}}
\newcommand{\taij}{\tilde a_{ij}}

\newcommand{\tbij}{\tilde b_{i|j}}
\newcommand{\haij}{\hat a_{ij}}

\newcommand{\hbij}{\hat b_{i|j}}
\newcommand{\baij}{\bar a_{ij}}

\newcommand{\bbij}{\bar b_{i|j}}

\newcommand{\G}{G^i_\alpha}
\newcommand{\tG}{\tilde G^i_{\tilde\alpha}}
\newcommand{\hG}{\hat G^i_{\hat\alpha}}
\newcommand{\bG}{\bar G^i_{\bar\alpha}}

\newcommand{\rij}{r_{ij}}
\newcommand{\sij}{s_{ij}}
\newcommand{\ri}{r_i}
\newcommand{\si}{s_i}
\newcommand{\rj}{r_j}
\newcommand{\sj}{s_j}

\begin{document}
\title{Deformations and Hilbert's Fourth Problem}
%\footnotetext{\emph{Keywords}: Finsler metric, $(\alpha,\beta)$-metric, projective flatness, deformation.
%\\
%\emph{Mathematics Subject Classification}: 53B40, 53C60.}}
\author{Changtao Yu}
\date{}
\maketitle

\begin{abstract}
In this paper we study a class of Finsler metrics defined
by a Riemannian metric and an 1-form. We classify those of
projectively flat in dimension $n\geq3$ by a special class of deformations. The results show that the projective flatness of such kind of Finsler metrics always arises from that of some Riemannian metric.
\end{abstract}

\section{Introduction}
Hilbert's Fourth Problem asks to characterize all the intrinsic {\it quasimetrics}, namely those distance functions which satisfy all axioms for a metric with the possible exception of symmetry, on a subset in $\RR^n$ such that the straight line segment are the shortest paths\cite{Hi}. Since any intrinsic quasimetric induces a Finsler metric, Hilbert's Fourth Problem in the regular case is to study the Finsler metrics with straight lines as their geodesics, such kind of Finsler metrics is said to be {\it projectively flat}. This problem has been solved in Riemannian geometry by Beltrami. However, it is far from been solved for general Finsler metrics. The main results are given by Busemann and Pogorelov based on integral geometric\cite{Bu,Po} and Paiva based on symplectic geometry\cite{Pa}. But all of their research are just for the {\it absolutely homogeneous} Finsler metrics, which means that the distance functions determined by the metrics are {\it symmetric}.

The notion of projectively flatness is closely connected with the curvature properties of a Finsler metric. It is known that a Finsler metric is {\it locally} projectively flat if and only if $F$ is of vanishing Douglas curvature and Weyl curvature, both of which are invariants in projective Finsler geometry. Moreover, the locally projectively flat Finsler metrics must be of scalar flag curvature. Flag curvature for a Finsler metric is an analogue of the sectional curvature in Riemannian geometry. A Finsler metric is said to be of {\it scalar} or {\it constant} flag curvature if the flag curvature $K(P,y)$ is a scalar function of tangent vectors $y$ or a constant. According to Beltrami's theorem, a Riemannian metric is locally projectively flat if and only if it is of constant sectional curvature. This property doesn't hold for Finsler metrics in general, but at least it is clear that locally projectively flat Finsler metrics include all the constant curvature metrics. There are many inspiring researches on the related topics\cite{Br1,Br2,szm-pffm}.

Although Busemann provided a specific way to construct projectively flat metrics, the properties and structure of these metrics are not clear enough\cite{Pa}. But if we just consider a smaller class of Finsler metrics, it is possible to obtain some more direct descriptions illuminating the underlying geometry.

In this paper, we will focus on a special class of Finsler metrics called {\it $\ab$-metrics}. They are those Finsler metrics which are closest to Riemannian metrics in a sense. This kind of metrics are expressed as $F=\pab$ in terms of a Riemannian metric $\a=\sqrt{a_{ij}(x)y^iy^j}$, an $1$-form $\b=b_i(x)y^i$ and a smooth function $\ps$. $\ab$-metrics were proposed by M. Matsumoto as a generalization of Randers metrics, the later were first introduced by a physicist G. Randers in 1941 from the standpoint of general relativity.

In spite of the special form, $\ab$-metrics are natural in geometrical point of view. As we known that a Minkowski norm on a $n$-dimensional vector space is Euclidean if and only if it is preserved under the action of $O(n)$. The author and H. Zhu prove in \cite{YCT} that $\ab$-norms are those Minkowski norms which are preserved under the action of $O(n-1)$. In other words, the {\it indicatrix} of an $\ab$-norm is a rotation hypersurface with the rotation axis passing the origin. If the hypersurface is a hyperelliptic, then the corresponding norm is a Randers norm. And if the origin is just the center of the hyperelliptic additionally, then the norm is Euclidean. For an $\ab$-norm, the function $\ps$ includes all the geometrical information of the indicatrix.

Compared to general Finsler metrics, $\ab$-metrics has wonderful computability(see \cite{bacs-cxy-szm-curv} and references therein). Hence, the researches on $\ab$-metrics enrich Finsler geometry and the approaches offer references for further study. Some basic facts $\ab$-metrics are provided in Section \ref{basic}.

As the simplest kind of $\ab$-metrics in form, Randers metrics are expressed as $F=\a+\b$ with $\|\b\|_\a<1$. It is known that a Randers metric $F=\a+\b$ is locally projectively flat if and only if $\a$ is locally projectively flat,which implies $\a$ is of constant sectional curvature by Beltrami's theorem, and $\b$ is closed\cite{BM}. For instance, the well-known {\it Funk metric}
\begin{eqnarray}\label{funk}
F=\f{\sqrt{(1-|x|^2)|y|^2+\langle x,y\rangle^2}}{1-|x|^2}-\f{\langle x,y\rangle}{1-|x|^2}
\end{eqnarray}
is projectively flat on the unit ball $\mathbb B^n(1)$ with constant flag curvature $K=-\f{1}{4}$.

Another important kind of $\ab$-metrics was given below by L.~Berwald\cite{Be},
\begin{eqnarray}\label{BerwaldF}
F=\f{(\sqrt{(1-|x|^2)|y|^2+\langle x,y\rangle^2}+\langle x,y\rangle)^2}{(1-|x|^2)^2\sqrt{(1-|x|^2)|y|^2+\langle x,y\rangle^2}}.
\end{eqnarray}
It belongs to a special kind of $\ab$-metrics given in the form $F=\f{(\a+\b)^2}{\a}$ with $\|\b\|_\a<1$. Berwald's metric is also projectively flat on $\mathbb B^n(1)$ with vanishing flag curvature $K=0$.

Recent research given by D. Bao et al. shows the internal relationship between Randers metrics and Zermelo's navigation problem on Riemannian spaces\cite{db-robl-szm-zerm}. Consider a Riemannian manifold $(M,h)$ and a vector field $W$ with $|W|_h:=\sqrt{h(W,W)}<1$, then the paths of shortest travel time in the Riemannian manifold $(M,h)$ under the influence of the wind $W$ are just the geodesics of a Randers metric $F=\a+\b$, where $\a$ and $\b$ are determined by $h$ and $W$ as $$\a=\f{\sqrt{(1-|W|_h^2)h^2+(W^\flat)^2}}{1-|W|_h^2},\quad\b=-\f{W^\flat}{1-|W|_h^2}.$$
Conversely, the navigation data $(h,W)$ are given by
$$h=\sqrt{1-b^2}\sqrt{\a^2-\b^2},\quad W^\flat=-(1-b^2)\b,$$
where $b:=\|\b\|_\a$ is the norm of $\b$ with respect to $\a$. For the Funk metric, one can see that $h=|y|$ is just the standard Euclidian metric and $W=x^i\frac{\partial}{\partial x^i}$.

By using the navigation description for Randers metrics, the main result in \cite{db-robl-szm-zerm} says that a Randers metric is of constant flag curvature if and only if $h$ is of constant sectional curvature and $W$ is an infinitesimal homothety of $h$, which indicates {\it the constancy of the curvature for a Randers metric always arises from that of some Riemannian metric by doing the navigation deformations}.

The main purpose of these paper is to provide a luminous characterization for locally projectively flat $\ab$-metrics. One of our main results shows that the answer is much similar to the phenomenon for Randers metrics with constant curvature: {\it the projective flatness of an $\ab$-metric always arises from that of some Riemannian metric when dimension greater than 2}.

We will firstly discuss Berwald's metrics $F=\sq$ in Section \ref{sq} as the special case. There are two reasons to do so. Firstly, except for Randers metrics, many non-trivial results in Finsler geometry relate to Berwald's metrics. Both Randers metrics and Berwald's metrics are the rare Finsler metrics of excellent geometry properties\cite{LB,SSZ,SY}. It seems that such metrics play a particular role in Finsler geometry. Secondly, the concise argument for Berwald's metrics sheds light on the general case. In fact, the projective flatness of Berwald's metrics had been studied in \cite{MSY}. But due to the inherent limitation of the method used in it, the argument is insufficient. However, it inspires the author to develop a new class of metric deformations called {\it $\b$-deformations}.

$\b$-deformations, determined by a Riemannian metric $\a$ and an $1$-form $\b$, are the key for our question. Some basic properties are given in Section \ref{bdf}. We can offer a brief description for the role of $\b$-deformations here. Consider some property of a given $\ab$-metric. Generally it is equivalent to some conditions on $\a$ and $\b$. But the geometry meaning of the original data $(\a,\b)$ is very obscure frequently. So the method of $\b$-deformations is aim to make clear the latent light. For an analogy, $\a$ and $\b$ just like two ropes tangles together, and it is possible to unhitch them using $\b$-deformations. The navigation expression for Randers metrics is a representative example. In fact, it is just a specific kind of $\b$-deformations. In other words, $\b$-deformations can be regarded as a natural generalization of navigation expression for Randers metrics.

Back to Berwald's metrics, we provide two kinds of $\b$-deformations for them, one is (\ref{Berwaldepsimple}) in Section \ref{sq} and the other is (\ref{Berwaldepsimple2}) in Section \ref{5}. The further research shows that these two expressions are exactly tailored for Berwald's metrics\cite{SY}, just as the navigation expression for Randers metrics. But it is unknown that whether there is a suitable physical or geometric explanation for them.

In view of the importance of the navigation expression for Randers metrics, we have reason to believe that $\b$-deformations will become a basic tool in the future study. We're really looking forward to see that more and more applications of this new class of metric deformations can be found in Finsler geometry, especially in Riemann geometry. Actually, something had been done in the author's doctoral dissertation, and the contents of this article are the representative part in it.

%Recently, Z. Shen obtain a characterization for $\ab$-metrics to be locally projectively flat in dimension greater than $2$. Roughly speaking, for a {\it non-trivial} projectively flat $\ab$-metric, there are some constraints on $\a$, $\b$ and $\ps$. More detailed statement is given in Section \ref{basic}.

In Section \ref{5} we will discuss the non-trivial case and prove one of the main results, namely Theorem \ref{main}. Note that if $\a$ is local projectively flat and $\b$ is parallel with respect to $\a$, then $F=\pab$ is inevitably local projectively flat for any suitable function $\ps$ (see Section \ref{basic} for the reason). This is what we say the trivial case. The relevant functions and some some explicit examples are determined in Section \ref{6}. As a result, we obtain the following interesting metrics: take some suitable value of $\epsilon$ such that the functions
\begin{eqnarray}\label{exampleonball}
\p_\sigma=1+\epsilon s+\sum_{n=1}^{\infty}\left\{\prod_{k=1}^{n}\f{(k-\sigma-1)(2k-3)}{k(2k-1)}\right\}s^{2n}
\end{eqnarray}
are positive in the interval $(-1,1)$, then the Finsler metrics
$$F=\f{\sqrt{(1-\xx)\yy+\xy^2}}{(1-\xx)^{\sigma+1}}\p_\sigma\left(\f{\xy}{\sqrt{(1-\xx)\yy+\xy^2}}\right)$$
are projectively flat on $\mathbb B^n(1)$. In particular, it is just the Funk metric (\ref{funk}) when $\sigma=0$ and $\epsilon=1$, and the Berwald's metric (\ref{BerwaldF}) when $\sigma=1$ and $\epsilon=2$.

Due to the non-uniqueness of expressions, many metrics in Theorem \ref{main} are essentially of same {\it type} (see Definition \ref{type}). The related questions are discussed in Section \ref{non-uniqueness} and the final classification is given by Theorem \ref{projab}.

On the other hand, a basic fact proved by Rapcs\'{a}k says that if an $1$-form $\theta$ is closed, then two Finsler metrics $F$ and $F+\theta$ are {\it pointwise projectively related}, i.e., the have the same geodesics as point sets. Hence, if $F$ is a {\it reversible} locally Finsler metric, then we can obtain infinity many irreversible locally projectively flat Fisler metrics by plus closed $1$-forms (of course we should assume the resulting function can be as a Finsler metric). Theorem \ref{pr} will show that all the non-trivial irreversible locally projectively flat $\ab$-metrics can be obtained by this way. In this sense, there is not any {essentially} {\it irreversible} locally projectively flat $\ab$-metric. At the same time, we will show that there are one-parameter different types of reversible $\ab$-metrics can be locally projectively flat in the non-trivial sense, the simplest one in which is Riemannian metrics.

Finally, from Theorem \ref{main} we know that projectively flat Riemannian metrics and the closed $1$-forms which are also conformal are basic for our question. So in Section \ref{7} we will offer the complete list of them. As we have pointed out that the projectively flatness of $\ab$-metrics always arises form that of some Riemannian metric, we can else make an additional remark that it is such kinds of $1$-forms to preserve the projective flatness in fact.

\section{Preliminaries}\label{basic}

Let $M$ be a smooth $n$-dimensional manifold. A Finsler metric $F$ on $M$ is a continuous function
$F:TM\to[0,+\infty)$ with the following properties:
\begin{enumerate}
\item {\it Regularity}: $F$ is $C^\infty$ on the entire slit tangent bundle $TM\backslash\{0\}$;
\item {\it Positive homogeneity}: $F(x,\lambda y)=\lambda F(x,y)$ for all $\lambda>0$;
\item {\it Strong convexity}: the fundamental tensor $g_{ij}:=[\frac{1}{2}F^2]_{y^iy^j}$ is positive definite for all $(x,y)\in TM\backslash\{0\}$.
\end{enumerate}
Here $x=(x^i)$ denotes the coordinates of the point in $M$ and $y=(y^i)$ denotes the coordinates of the vector in $T_xM$. If $F$ satisfies absolute homogeneity additionally, namely $F(x,\lambda y)=|\lambda|F(x,y)$ for any real number $\lambda$, then $F$ is said to be {\it reversible}.

For a Finsler metric, the {\it geodesics} are characterized by the geodesic equation
$$\ddot c^i(t)+2G^i\left(c(t),\dot c(t)\right)=0,$$
where
$$G^i(x,y)=\f{1}{4}g^{il}\left\{[F^2]_{x^ky^l}y^k-[F^2]_{x^l}\right\}$$
are called the {\it spray coefficients} of $F$. Here $(g^{ij}):=(g_{ij})^{-1}$. For a Riemannian metric $\a$, its spray coefficients are given by
$$G^i(x,y)=\f{1}{2}\Gamma^i{}_{jk}(x)y^jy^k$$
in terms of the Christoffel symbols of $\a$.

A Finsler metric $F$ on a manifold $M$ is said to be {\it locally projectively flat} if at any point of $M$, there is a local coordinate system in which all geodesics of $F$ are straight lines. A basic fact is that $F$ is projectively flat on an open subset $U\subseteq\RR^n$ if and only if $F$ satisfies Rapcs\'ak's equations $$F_{x^ky^l}y^k-F_{x^l}=0,$$
or equivalently Hamel's equations
$$F_{x^ky^l}=F_{x^ly^k}.$$
In this case, the spray coefficients of $F$ are given by
\begin{eqnarray}\label{GPI}
G^i=Py^i,
\end{eqnarray}
where $P=\frac{F_{x^k}y^k}{2F}$\cite{Ha,Ra}. In fact, (\ref{GPI}) is also a characterization for projectively flat Finsler metrics. It will be used frequently in our discussion.

Suppose that $\phi(s)$ is a {\it positive} smooth function in some symmetric open interval $(-b_o,b_o)$. Then the function $F=\pab$ is a Finsler metric for any Riemannian metric $\a=\sqrt{a_{ij}(x)y^iy^j}$ and $1$-form $\b=b_i(x)y^i$ if and only if for all the reals $s$ and $b$ satisfying $|s|\leq b<b_o$, the following inequality holds\cite{CS2}:
\begin{eqnarray}\label{conditionofphi}
\p(s)-s\p'(s)+(b^2-s^2)\p''(s)>0.
\end{eqnarray}
For a given metric, $b$ means the norm of $\b$, which is defined by
$$b:=\sup\limits_{y\in T_xM}\f{\b(x,y)}{\a(x,y)}=\sqrt{a^{ij}(x)b_i(x)b_j(x)}.$$

This kind of Finsler metrics is the so-call {\it $\ab$-metrics}. We can get Randers metrics by taking $\ps=1+s$ and Berwald's metrics by taking $\ps=(1+s)^2$. But one must be careful that there are infinity many expressions for any given $\ab$-metric. For example, for all the functions $\ps=\sqrt{C_1+C_2s^2}+C_3s$, in which $C_1,C_2,C_2$ are arbitrary constants, the corresponding $\ab$-metrics $F=\sqrt{C_1\a^2+C_2\b^2}+C_3\b$ are all Randers metrics. This non-uniqueness will be useful in our discussions (see Section \ref{non-uniqueness} for details).

For a given $\ab$-metric $F=\pab$, let $\bij$ be the coefficients of the covariant derivative of
$\b$ with respect to $\a$. Denote $\rij$ and $\sij$ be symmetrization and antisymmetrization of $\bij$ respectively, i.e.,
$$\rij:=\f{1}{2}(\bij+b_{j|i}),\qquad\sij:=\f{1}{2}(\bij-b_{j|i}).$$
Since $\bij-b_{j|i}=\frac{\partial b_i}{\partial x^j}-\pp{b_j}{x^i}$, $\sij=0$ implies $\b$ is closed, vice versa. Moreover, we need the following abbreviations,
\begin{eqnarray*}
&r_{00}:=r_{ij}y^iy^j,~r_i:=r_{ij}y^j,~r_0:=r_iy^i,~r:=r_ib^i,\\
&s_{i0}:=s_{ij}y^j,~s^i{}_0:=a^{ij}s_{j0},~s_i:=s_{ij}y^j,~s_0:=s_ib^i.
\end{eqnarray*}
Roughly speaking, indices are raised and lowered by $\aij$, vanished by contracted with $b^i$ and changed to `${}_0$' by contracted with $y^i$.

According to \cite{CS2}, the spray coefficients $G^i$ of the $\ab$-metric $F=\pab$ are related to that of $\a$ by
\begin{eqnarray}\label{GF}
G^i= \G+\a Q s^i{}_0+\a^{-1}\Theta(-2\a Q s_0+r_{00})y^i+\Psi(-2\a Q s_0+r_{00})b^i,
\end{eqnarray}
where
\begin{eqnarray*}
Q=\frac{\p'}{\p-s\p'},~\Theta=\frac{(\p-s\p')\p'-s\p\p''}{2\p\big(\p-s\p'+(b^2-s^2)\p''\big)},~
\Psi=\frac{\p''}{2\big(\p-s\p'+(b^2-s^2)\p''\big)}.
\end{eqnarray*}

If $\a$ is locally projectively flat and $\b$ is parallel with respect to $\a$, then one can see by the above formula that $G^i=\G=Py^i$ for some suitable local coordinate system, which means that the $\ab$-metrics are always locally projectively flat for all the suitable functions $\ps$. Actually, such a Finsler metric must be either a Riemannian metric with constant sectional curvature by Beltrami's result when $\b=0$ or a local Minkowski metric when $\b\neq0$. For the later, $\a$ must be flat. Recall that a Finsler metric $F$ on an open subset in $U\subseteq\RR^n$ is said to be {\it Minkowskian} if the values of $F$ are independent of the points in $U$, just like the standard Euclidian metric in Riemannian metric. In a word, it can be regarded trivial when $\a$ is locally projectively flat and $\b$ is parallel with respect to $\a$.

A result of Z. Shen et al.\cite{SG} says that a Berwald's metric $F=\sq$ is locally projectively flat if and only if the spray coefficients of $\a$ are given in a adapted coordinate system by
\begin{eqnarray}\label{sqG}
\G=\xi y^i-2\tau\a^2b^i
\end{eqnarray}
for some $1$-form $\xi=\xi_i(x)y^i$ on the manifold and some scalar function $\tau=\tau(x)$, and at the same time the covariant derivative of $\b$ is given by
\begin{eqnarray}\label{sqbij}
\bij=2\tau\left\{(1+2b^2)\aij-3\bi\bj\right\}.
\end{eqnarray}

Later on, Z. Shen find a sufficient and necessary condition for $\ab$-metrics to be locally projectively flat in dimension $n\geq3$\cite{S}. It says that for a projectively flat $\ab$-metric $F=\pab$ on an open subset $U\in\RR^n$ with $n\geq3$, if we add three conditions:
\vspace{-2bp}
\begin{enumerate}[(a)]
\item {\it $F$ is not of Randers type;}
\item {\it $\b$ is not parallel with respect to $\a$;}
\item {\it $\ud b\equiv0$ or $\ud b\neq0$ everywhere;}
\end{enumerate}
then the corresponding function $\ps$ must be a solution of the following 2nd  ordinary differential equation,
\begin{eqnarray}\label{ode}
\left\{1+(k_1+k_3)s^2+k_2s^4\right\}\p''(s)=(k_1+k_2s^2)\{\ps-s\p'(s)\},
\end{eqnarray}
where $k_1,k_2,k_3$ are constants with $k_2\neq k_1k_3$, at the same time $\a$ and $\b$ satisfy
\begin{eqnarray}
\G&=&\xi y^i-\tau(k_1\a^2+k_2\b^2)b^i,\label{Gi}\\
\bij&=&2\tau\left\{(1+k_1b^2)a_{ij}+(k_3+k_2b^2)b_ib_j\right\}\label{bij}
\end{eqnarray}
for some $1$-form $\xi$ on $U$ and some scalar function $\tau$.

Notice that when $k_2=k_1k_3$, solve (\ref{ode}) we have $\ps=C_1\sqrt{1+k_1s^2}+C_2s$, in this case the corresponding $\ab$-metrics are Randers metrics essentially. So the condition $k_2\neq k_1k_3$ ensure (a) holds. The condition (b) is also natural, because if $\b$ is parallel with respect to $\a$, then by (\ref{GF}) $F=\pab$ is projectively flat if and only if $\a$ is projectively flat. This is the trivial case as we have pointed out earlier. The condition (c) is a little hard to been understood. It is completely technical. Actually, as Z. Shen told to the author, in order to get the sufficient and necessary condition, we only need to assume that $b$ is a constant or can take sufficient numbers of values.

Recently, \cite{YGJ} shows that, except for two {\it type} of $\ab$-metrics $F=\a\pm\frac{\b^2}{\a}$ (see Definition \ref{type} for the meaning of type), Z. Shen's sufficient and necessary condition is hold in dimension $n=2$. Note that for these exceptional metrics, the corresponding functions $\ps=1\pm s^2$ are still solves of (\ref{ode}).

%Another result also given by Z.~Shen tell us that if $\ps$ satisfies (\ref{ode}), then the $\ab$-metric $F=\pab$ is locally projectively flat in dimension $n\geq2$ if and only if both (\ref{Gi}) and (\ref{bij}) hold[\textcolor[rgb]{1.00,0.00,0.00}{??}].

Anyway, from (\ref{Gi}) and (\ref{bij}) one can see that the geometry properties of $\a$ and $\b$ are very elusive. $\a$ includes too much information about $\b$, so does $\b$. Hence, we need to separate them to make them clear. The following Lemma is needed for fulfilling the task later.
\begin{lemma}\label{condition2}
Let~$F=\pab$~be a Finlser metric, where
$\ps$ satisfies (\ref{ode}). Then the following inequalities hold:
$$1+k_1s^2>0,\quad 1+(k_1+k_3)s^2+k_2s^4>0,\quad\forall|s|\leq b<b_o.$$
\end{lemma}
\begin{proof}
By (\ref{ode}),
\begin{eqnarray*}
&&\ps-s\p'(s)+(b^2-s^2)\p''(s)\\
&=&\{\ps-s\p'(s)\}\left\{1+\f{(b^2-s^2)(k_1+k_2s^2)}{1+(k_1+k_3)s^2+k_2s^4}\right\}\\
&=&\{\ps-s\p'(s)\}\cdot\f{1+k_1b^2+s^2(k_3+k_2b^2)}{1+(k_1+k_3)s^2+k_2s^4}.
\end{eqnarray*}
By taking $b=s$ in (\ref{conditionofphi}), we can see that $\ps-s\p'(s)$ is always positive as long
\newpage
\noindent
as $F$ is a Finsler metric. So the condition (\ref{conditionofphi}) implies the following inequality
\begin{eqnarray}
\f{1+k_1b^2+s^2(k_3+k_2b^2)}{1+(k_1+k_3)s^2+k_2s^4}>0,\quad\forall|s|\leq b<b_o.\label{something}
\end{eqnarray}
Take $s=0$ we have $1+k_1b^2>0$, then the first inequality in the Lemma holds because $1+k_1s^2\geq\min\{1,1+k_1b^2\}$. The second inequality in the Lemma is also true because in order to keep the inequality (\ref{something}) hold, the denominator must be positive.
\end{proof}

\section{Berwald's metrics}\label{sq}

We will begin our discussion with a special case, namely how to describe the projective flatness for Berwald's metrics. It is very simple but illuminating. Recall that such a metric is projectively flat if and only if $\a$ and $\b$ satisfy (\ref{sqG}) and (\ref{sqbij}), but their geometric meaning are not so clear. Inspired by the argument in \cite{MSY}, we will take some deformations to make them simpler.

Consider $\a$ firstly. If we can take some deformations for it such that the spray coefficients of the resulting Riemannian metric $\ha$ doesn't include the term $b^i$, i.e., $\hG=\hat\xi y^i$ for some function $\hat\xi=\hat\xi(x)$, then $\ha$ is projectively flat. It is the best situation we can expect, and this aim can be achieved easily just by doing conformal deformation.

More specifically, let $\ha=e^{\rho(b^2)}\a$. Some elementary computations show
\begin{eqnarray*}
\hat\Gamma^i{}_{jk}%&=&\f{1}{2}\hat a^{il}\left(\pp{\hat a_{jl}}{x^k}+\pp{\hat a_{kl}}{x^j}-\pp{\hat a_{jk}}{x^l}\right)\\
=\Gamma^i{}_{jk}+2\rho'\left\{\delta^i{}_j(r_k+s_k)+\delta^i{}_k(r_j+s_j)-a_{jk}(r^i+s^i)\right\},
\end{eqnarray*}
here we use the fact that
$$\pp{b^2}{x^i}=2b^jb_{j|i}=2(r_i+s_i).$$
As a result, when $\a$ and $\b$ satisfy (\ref{sqG}) and (\ref{sqbij}), then
\begin{eqnarray}
\hG&=&\G+\rho'\left\{2(r_0+s_0)y^i-\a^2(r^i+s^i)\right\}\nonumber\\
&=&\left\{\theta+4\tau\rho'(1-b^2)\beta\right\}y^i-2\tau\left\{1+\rho'(1-b^2)\right\}\a^2b^i\label{temphG}
\end{eqnarray}
From the last equality one can see that $\ha$ is projectively flat if and only if
$$1+\rho'(1-b^2)=0,$$
so the conformal factor can be chosen as $\rho(b^2)=\ln(1-b^2)$. Notice that $b<1$ when $F=\sq$ is a Finsler metric.

Our first task is done. Consider $\b$ next. For convenience, let $\hat\b=\b$ and denote $\hbij$ be the covariant derivative of $\hb$ with respect to $\ha$. Then
\begin{eqnarray*}
\hbij&=&\pp{b_i}{x^j}-b_k\Gamma^k{}_{ij}+b_k(\Gamma^k{}_{ij}-\hat\Gamma^k{}_{ij})\\
&=&\bij-2\rho'\left\{\bi(r_j+s_j)+\bj(r_i+s_i)-\aij r\right\}\\
&=&2\tau(\aij+\bi\bj).
\end{eqnarray*}

In spite of the property of $\ha$ is good enough, $\hb$ is still not so good and clear. But we can take some deformations for it. Let $\bb=\nu(b^2)\hb$, and for convenience,
\newpage
\noindent let $\ba=\ha$ and denote $\bbij$ be the covariant derivative of $\bb$ with respect to $\ba$. Then $\ba$ is projectively flat and a direct computation shows that
\begin{eqnarray*}
\bbij&=&\nu\hbij+2\nu'\hbi(r_j+s_j)\\
&=&2\tau\left\{\nu\aij+\nu\bi\bj+2\nu'(1-b^2)\bi\bj\right\}.
\end{eqnarray*}
Form the above equality one can see that if the deformation factor $\nu$ satisfies
\begin{eqnarray}\label{sqnu}
\nu+2\nu'(1-b^2)=0,
\end{eqnarray}
then $\hbij$ will don't include the term $\bi\bj$ and given in the form $\bbij=c(x)\baij$.

We believe it is the best situation for the $1$-form. The main reason is that such kind of $1$-forms has great linear structure. More specifically, the linear combination of two different $1$-forms satisfying $\rij=c(x)\aij$ will keep the same property. In fact, the dual tensors of such kind of $1$-forms are just the conformal vector field with respect to the corresponding Riemannian metric. In the rest of the paper, such $1$-forms will be called {\it conformal $1$-forms}.

Back to the argument. By (\ref{sqnu}) the deformation factor can be chosen as $\nu(b^2)=\sqrt{1-b^2}$. So far we have prove that if a Finsler metric $F=\sq$ is locally projectively flat, then $\ba$ is locally projectively flat, $\bb$ is closed and conformal with respect to $\ba$.

\[\begin{CD}
\a @>\ha=(1-b^2)\a>> \ha @>\,\,\ \ \ba=\ha\ \ \ \,\,>> \ba~(\textrm{projectively flat})\ \ \ \ \ \ \ \\
\b @>\ \ \ \hb=\b\ \ \ \, >> \hb @>\bb=\sqrt{1-b^2}\hb>> \bb~(\textrm{closed and conformal})
\end{CD}\]
\vspace{4bp}

Conversely, the norm $\bar b:=\|\bb\|_{\ba}$ is related to $b$ as
$$\bar b^2=\bar a^{ij}\bbi\bbj=(1-b^2)^{-2}a^{ij}\sqrt{1-b^2}\bi\sqrt{1-b^2}\bj=\f{b^2}{1-b^2},$$
or equivalently, $(1-b^2)(1+\bar b^2)=1$. Hence, the above deformations are reversible. So one can easily obtain the required data satisfying (\ref{sqG}) and (\ref{sqbij}) by the reverse deformations.

\[\begin{CD}
\a @<\a=(1+\bar b^2)\ha<< \ha @<\,\,\ \ \ha=\ba\ \ \ \,\,<< \ba~(\textrm{projectively flat})\ \ \ \ \ \ \ \\
\b @<\ \ \ \b=\hb\ \ \ \, << \hb @<\hb=\sqrt{1+\bar b^2}\bb<< \bb~(\textrm{closed and conformal})
\end{CD}\]
\vspace{4bp}

Summarizing the above discussions, we have
\begin{theorem}\label{ab2}
The Finsler metric $F=\frac{(\a+\b)^2}{\a}$ is locally projectively flat if and only if $\a$ and $\b$ can be expressed as the following form,
$$\a=(1+\bar b^2)\ba,\qquad\b=\sqrt{1+\bar b^2}\bb,$$
where $\ba$ is a locally projectively flat Riemannian metric, $\bb$ is a closed 1-form and conformal with respect to $\ba$, $\bar b:=\|\bb\|_{\ba}$. In this case, $F$ can be rewrote as
\begin{eqnarray}\label{Berwaldepsimple}
F=\f{\left(\sqrt{1+\bar b^2}\ba+\bb\right)^2}{\ba}.
\end{eqnarray}
\end{theorem}

\section{A special class of deformations}\label{bdf}

One can see from the former section that deformations play an important role in our question. In this section, we will introduce a triple of metric deformations and provide the necessary formulas.

It should be pointed out firstly that the deformations used here are essentially different from those in \cite{MSY}. The most notable is, the deformations factors here are always the function of $b$, not only of points. In Section \ref{sq} we have fully show their effectiveness. Such deformations make the whole argument concise.

Even so, only conformal deformations are not enough to make the Riemannian metrics projectively flat for the general case. The reason is, if $\a$ and $\beta$ satisfy (\ref{Gi}) and (\ref{bij}), then by the conformal deformation $\ha=e^{\rho(b^2)}\a$ we can only dispel the term $\a^2b^i$ (see the formula (\ref{temphG})), just as what we have done in Section \ref{sq}, but it is impossible to dispel the term $\b^2b^i$ when $k_2\neq0$.

Fortunately, there is still one kind (it is also the last kind in fact) of deformations could be chosen if we insist using the functions of $b$ as deformation factors. It is given by $\a^2-\kappa(b^2)\b^2$ with the condition $1-\kappa b^2>0$ to keep it positive definite.

So far, we have three kinds of deformations in terms of $\a$ and $\b$. Let's list them in turn as following:
\begin{eqnarray*}
&\ta=\sqrt{\a^2-\kappa(b^2)\b^2},\qquad\tb=\b;\\
&\ha=e^{\rho(b^2)}\ta,\qquad\hb=\tb;\\
&\ba=\ha,\qquad\bb=\nu(b^2)\hb.
\end{eqnarray*}
Here we choose $b^2$ instead of $b$ as the variable, because it will be convenient for computations.

According to their characteristics, we call them {\it $\b$-deformations}. Obviously, the first two kinds of $\b$-deformations are for Riemannian metrics and the last one is for $1$-forms. More specifically, the first kind of $\b$-deformation can be regarded as stretch change for $\a$ along the direction determined by $\b$, the second one is conformal change and the third one is length change for $\b$.

Some basic formulas are listed below and the elementary proofs are left out here. It should be attention that the notation `$\dot b_{i|j}$' always means the covariant derivative of the $1$-form `$\dot\b$' with respect to the {\it corresponding} Riemannian metric `$\dot\a$', where the symbol `~$\dot{}$~' can be `~$\tilde{}$~', `~$\hat{}$~' or `~$\bar{}$~' in this paper.
\begin{lemma}\label{beta1}
Let $\ta=\sqrt{\a^2-\kappa(b^2)\b^2}$, $\tb=\b$. Then
\begin{eqnarray*}
\tG&=&\G-\frac{\kappa}{2(1-\kappa b^2)}\big\{2(1-\kappa b^2)\b s^i{}_0+r_{00}b^i+2\kappa s_0\b b^i\big\}\\
&&+\frac{\kappa'}{2(1-\kappa b^2)}\big\{(1-\kappa b^2)\b^2(r^i+s^i)+\kappa r\b^2b^i-2(r_0+s_0)\b
b^i\big\},\\
\tbij&=&\bij+\frac{\kappa}{1-\kappa b^2}\big\{b^2\rij+\bi\sj+\bj\si\big\}\\
&&-\frac{\kappa'}{1-\kappa b^2}\big\{r\bi\bj-b^2\bi(\rj+\sj)-b^2\bj(\ri+\si)\big\}.
\end{eqnarray*}
\end{lemma}

\begin{lemma}\label{beta2}
Let $\ha=e^{\rho(b^2)}\ta$, $\hb=\tb$. Then
\begin{eqnarray*}
\hG&=&\tG+\rho'\left\{2(r_0+s_0)y^i-(\a^2-\kappa \b^2)\left(r^i+s^i+\frac{\kappa}{1-\kappa b^2}rb^i\right)\right\},\\
\hbij&=&\tbij-2\rho'\left\{\bi(\rj+\sj)+\bj(\ri+\si)-\frac{1}{1-\kappa b^2}r(\aij-\kappa\bi\bj)\right\}.
\end{eqnarray*}
\end{lemma}

\begin{lemma}\label{beta3}
Let $\ba=\ha$, $\bb=\nu(b^2)\hb$. Then
\begin{eqnarray*}
\bG&=&\hG,\\
\bbij&=&\nu\hbij+2\nu'\bi(\rj+\sj).
\end{eqnarray*}
\end{lemma}

\section{General case}\label{5}
Suppose that $F=\pab$ is a non-trivial projectively flat $\ab$-metric on $U\subseteq\RR^n$ with $n\geq3$, then $\a$ and $\b$ satisfy (\ref{Gi}) and (\ref{bij}) according to Z. Shen's result. As what we have point out previously, it is impossible to make $\a$ projectively flat directly only using conformal deformations when $k_2\neq0$. Even if it could be done when $k_2=0$, we aim to search an universal deformation way.

There is a little difficulty to decide the first step of $\b$-deformations. But combining with Lemma \ref{beta2} one can see that, in order to make $\ta$ turning to a projectively flat metric $\ha$, which means that the spray coefficients of $\ha$ are in the form $\hG=\hat P y^i$, the spray coefficients of $\ta$ must possess a character that the coefficient of $b^i$ is scalar for $\ta^2$. This observation leads to the following
\begin{lemma}\label{step1}
Assume that $\a$ and $\b$ satisfy (\ref{Gi}) and (\ref{bij}). Take the first step of $\b$-deformations. Then the spray coefficients of $\ta$ are in the form
\begin{eqnarray}\label{tGi}
\tG=\xi y^i-\tilde\tau\ta^2b^i
\end{eqnarray}
if and only if the deformation factor $\kappa(b^2)$ satisfies
\begin{eqnarray}
\left\{1+(k_1+k_3)b^2+k_2b^4\right\}\kappa'+\kappa^2+(k_1+k_3)\kappa+k_2=0.\label{kappa}
\end{eqnarray}
Specially, $\kappa(b^2)=-(k_1+k_3+k_2b^2)$ is a solution of (\ref{kappa}) and in this case,
\begin{eqnarray}
\tG&=&\xi y^i+\f{\tau(k_3+k_2b^2)}{1+(k_1+k_3)b^2+k_2b^4}\ta^2b^i,\label{tGi'}\\
\tbij&=&2\tau\left\{\f{1+k_1 b^2}{1+(k_1+k_3)b^2+k_2b^4}\taij-(k_1+k_2b^2)\bi\bj\right\}.\label{tbij'}
\end{eqnarray}
\end{lemma}
\begin{proof}
By Lemma \ref{beta1}, the formula of $\tG$ are given by
\begin{eqnarray*}
\tG&=&\xi
y^i-\f{\tau}{1-\kappa b^2}\big\{(k_1+\kappa)\a^2\\
&&+\left[k_2+k_3\kappa+\kappa'(1+(k_1+k_3)b^2+k_2b^4)\right]\b^2\big\}b^i.\nonumber
\end{eqnarray*}
Obviously, $\tG$ is in the form (\ref{tGi}) if and only if
$$-\kappa(k_1+\kappa)=k_2+k_3\kappa+\kappa'\left\{1+(k_1+k_3)b^2+k_2b^4\right\},$$
which is equivalent to (\ref{kappa}). It is easy to verify that $\kappa=-(k_1+k_3+k_2b^2)$ is a
\newpage
\noindent solution of
(\ref{kappa}), and it can be chosen as a deformation factor since
$$1-\kappa b^2=1+(k_1+k_3)b^2+k_2b^4>0$$
by Lemma \ref{condition2}. By Lemma \ref{beta1} again we obtain (\ref{tGi'}) and (\ref{tbij'}). Notice that we use the fact $\taij=\aij-u\bi\bj$ here.
\end{proof}

\begin{lemma}\label{step2}
Choose the factor of the second step of $\b$-deformations as
$$\rho(b^2)=\int\f{(k_3+k_2b^2){\mbox d}\,b^2}{2\{1+(k_1+k_3)b^2+k_2b^4\}}.$$
Then
\begin{eqnarray}
\hG&=&\left\{\xi+2\tau(k_3+k_2b^2)\b\right\}y^i,\label{hGi'}\\
\hbij&=&2\tau\left\{e^{-2\rho(b^2)}\hat a_{ij}-(k_1+2k_3+3k_2b^2)b_ib_j\right\}.\label{hbij'}
\end{eqnarray}
In particular, $\ha$ is projectively flat.
\end{lemma}
\begin{proof}
By Lemma \ref{beta2}, the formula of $\hG$ are given by
$$\hG=(\xi+2\rho'r_0)y^i+\tau\left\{\f{(k_3+k_2b^2)}{1+(k_1+k_3)b^2+k_2b^4}-2\rho'\right\}\ta^2b^i,$$
so $\ha$ is projectively flat if and only if
\begin{eqnarray}\label{rho'}
\rho'=\f{k_3+k_2b^2}{2\left\{1+(k_1+k_3)b^2+k_2b^4\right\}}.\label{rho2}
\end{eqnarray}
By Lemma \ref{beta2} again we obtain (\ref{hGi'}) and (\ref{hbij'}). Notice that we use the fact $\haij=e^{2\rho}\taij$ here.
\end{proof}

Now, we can finish the whole deformation procedure by taking the third step of $\b$-deformations, which is aim to make the $1$-form conformal, just as we have done in Section \ref{sq} for Berwald's metrics.
\begin{lemma}\label{step3}
Choose the factor of the third step of $\b$-deformations as $$\nu(b^2)=e^{\rho(b^2)}\sqrt{1+(k_1+k_3)b^2+k_2b^4}.$$
Then
\begin{eqnarray*}
\bG&=&\left\{\xi+2\tau(k_3+k_2b^2)\b\right\}y^i,\\
\bbij&=&2\tau e^{-\rho}\sqrt{1+(k_1+k_3)b^2+k_2b^4}\bar a_{ij}.
\end{eqnarray*}
In particular, $\bb$ is closed and conformal with respect to $\ba$.
\end{lemma}
\begin{proof}
By Lemma \ref{beta3}, $\ba$ is still projectively flat and the covariant derivative of $\hbij$ is given by
\begin{eqnarray*}
\bbij=2\tau\left\{\nu e^{-2\rho}\baij-\left[\nu(k_1+2k_3+3k_2b^2)-2\nu'\left(1+(k_1+k_3)b^2+k_2b^4\right)\right]\bi\bj\right\}.
\end{eqnarray*}
Obviously that
$$\bbij=c(x)\bar a_{ij}$$
if and only if
\begin{eqnarray*}
\f{\nu'}{\nu}&=&\f{k_1+2k_3+3k_2b^2}{2\left\{1+(k_1+k_3)b^2+k_2b^4\right\}}
=\f{k_1+k_3+2k_2b^2}{2\left\{1+(k_1+k_3)b^2+k_2b^4\right\}}+\rho'.
\end{eqnarray*}
\newpage
\noindent The solutions of the above equation are
$$\nu=Ce^{\rho(b^2)}\sqrt{1+(k_1+k_3)b^2+k_2b^4},$$
without loss of generality we can take $C=1$.
\end{proof}

Until now the discussions can be summarized as the diagram below.

\[\begin{CD}
\a @>\ta=\sqrt{\a^2-\kappa(b^2)\b^2}>> \ta @>\ha=e^{\rho(b^2)}\ta>> \ha @>\,\ \ \ba=\ha\ \ \ >> \ba~(\textrm{projectively flat})\ \ \ \ \ \ \ \\
\b @>\ \ \ \ \ \ \ \tb=\b\ \ \ \ \ \ \ >> \tb @>\ \ \ \hb=\tb\ \ \ \, >> \hb @>\bb=\nu(b^2)\hb>> \bb~(\textrm{closed and conformal})
\end{CD}\]
\vspace{5bp}

On the other hand, it is easy to verify that the inverse of $\taij$ is given by $\tilde a^{ij}=a^{ij}+\f{\kappa}{1-\kappa b^2}b^ib^j$, so
\begin{eqnarray*}
\bar b^2:=\|\bb\|^2_{\ba}=\nu^2\|\hb\|^2_{\ha}=\nu^2e^{-2\rho}\|\tb\|^2_{\ta}
=\nu^2e^{-2\rho}b_i\left(a^{ij}+\f{\kappa}{1-\kappa b^2}b^ib^j\right)b_j=b^2.
\end{eqnarray*}
For this reason, the whole process of the above $\b$-deformations is invertible. Hence, one can obtain the data $\a$ and $\b$ satisfying (\ref{Gi}) and (\ref{bij}) by taking the invert $\b$-deformations. See the diagram below.

\[\begin{CD}
\a @<\a=\sqrt{\ta^2+\kappa(\bar b^2)\tb^2}<< \ta @<\ta=e^{-\rho(\bar b^2)}\ha<< \ha @<\,\ \ \ \ha=\ba\ \ \ \ \,<< \ba~(\textrm{projectively flat})\ \ \ \ \ \ \ \\
\b @<\ \ \ \ \ \ \ \b=\tb\ \ \ \ \ \ \ << \tb @<\ \ \ \ \tb=\hb\ \ \ \ << \hb @<\hb=\nu^{-1}(\bar b^2)\bb<< \bb~(\textrm{closed and conformal})
\end{CD}\]
\vspace{5bp}
%$$\a=e^{-\rho(\bar b^2)}\sqrt{\ba^2+\kappa(\bar b^2)\nu^{-2}(\bar b^2)\bb^2},\quad\b=\nu^{-1}(\bar b^2)\bb,$$
%So the corresponding $\ab$-metric $F=\pab$ is locally projectively flat due to Z.Shen's result.

Note that by (\ref{rho'}) the deformation factor $\eta(\bar b^2):=e^{-\rho(\bar b^2)}$ can be chosen as
\begin{eqnarray}\label{etab}
\eta(\bar b^2)=\exp\left\{-\int_0^{\bar b^2}\f{k_3+k_2t}{2\{1+(k_1+k_3)t+k_2t^2\}}\ud t\right\},
\end{eqnarray}
which can be express as elementary functions. So we have
\begin{theorem}\label{main}
Let $F=\pab$ be a Finsler metric on a $n$-dimensional manifold $M$ with $n\geq3$, where the function $\ps$ satisfying (\ref{ode}). Then $F$ is locally projectively flat if and only if $\a$ and $\b$ can be expressed as
\begin{eqnarray*}
\a&=&\eta(\bar b^2)\sqrt{\ba^2-\f{(k_1+k_3+k_2\bar
b^2)}{1+(k_1+k_3)\bar
b^2+k_2\bar b^4}\bb^2},\\
\b&=&\f{\eta(\bar b^2)}{\left\{1+(k_1+k_3)\bar b^2+k_2\bar
b^4\right\}^\f{1}{2}}\bb,
\end{eqnarray*}
where $\ba$ is a locally projectively flat Riemannian metric, $\bb$ is a closed 1-form and conformal with respect to $\ba$, $\bar b:=\|\bb\|_{\ba}$.
The deformation factor $\eta(\bar b^2)$ is determined by the coefficients $k_1,k_2,k_3$ and given in the following five case,
\begin{enumerate}
\item When $k_2=0,~k_1+k_3=0$,
$$\eta(\bar b^2)=\exp\left\{-\f{k_3\bar b^2}{2}\right\};$$
\newpage
\item When $k_2=0,~k_1+k_3\neq0$,
$$\eta(\bar b^2)=\left\{1+(k_1+k_3)\bar b^2\right\}^{-\f{k_3}{2(k_1+k_3)}};$$
\item When $k_2\neq0,~\Delta_1>0$,
$$\eta(\bar b^2)=\f{\left\{\f{\sqrt{\Delta_1}+k_1+k_3}{\sqrt{\Delta_1-k_1-k_3}}\cdot\f{\sqrt\Delta_1-k_1-k_3-2k_2\bar b^2}{\sqrt\Delta_1+k_1+k_3+2k_2\bar
b^2}\right\}^\f{k_1-k_3}{4\sqrt\Delta_1}}{\sqrt[4]{1+(k_1+k_3)\bar
b^2+k_2\bar b^4}};$$
\item When $k_2\neq0,~\Delta_1=0$,
$$\eta(\bar b^2)=\f{\sqrt{2}\exp\left\{\f{k_3-k_1}{k_1+k_3}\left[\f{1}{2+(k_1+k_3)\bar b^2}-\f{1}{2}\right]\right\}}{\sqrt{2+(k_1+k_3)\bar b^2}};$$
\item When $k_2\neq0,~\Delta_1<0$,
$$\eta(\bar b^2)=\f{\exp\left\{\f{k_1-k_3}{2\sqrt{-\Delta_1}}\left(\arctan\f{k_1+k_3+2k_2\bar
b^2}{\sqrt{-\Delta_1}}-\arctan\f{k_1+k_3}{\sqrt{-\Delta_1}}\right)\right\}}{\sqrt[4]{1+(k_1+k_3)\bar b^2+k_2\bar b^4}},$$
\end{enumerate}
where $\Delta_1:=(k_1+k_3)^2-4k_2$.
\end{theorem}

Actually, if $\a$ $\b$ and $\ps$ satisfy the given conditions in Theorem \ref{main}, then the Finsler metric is projectively flat for any dimension. That is because in this case the spray coefficients of $F$ is in the form $G^i=Py^i$ by (\ref{GF}). So the sufficiency of Theorem \ref{main} is also true when $n=2$.

On the other hand, for a given non-trivial projectively flat $\ab$-metric, the suitable deformation is possibly non-unique. Take Berwald's metrics for example, the corresponding function of the Berwald's metrics $F=\sq$ is $\ps=(1+s)^2$, which is a solution of (\ref{ode}) with $k_1=2$, $k_2=0$ and $k_3=-3$. So by the above discussions, we prove that $F$ is locally projectively flat if and only if $\ba=(1-b^2)^\frac{3}{2}\sqrt{\a^2-\b^2}$ is locally projectively flat and $\bb=(1-b^2)^2\b$ is closed and conformal with respect to $\ba$. In this case, the invert deformations are given by $\a=(1-\bar b^2)^{-2}\sqrt{(1-\bar b^2)\ba^2+\bb^2}$ and $\b=(1-\bar b^2)^{-2}\bb$, hence $F$ can be rewrote as
\begin{eqnarray}\label{Berwaldepsimple2}
F=\f{(\sqrt{(1-\bar b^2)\ba^2+\bb^2}+\bb)^2}{(1-\bar b^2)^2\sqrt{(1-\bar b^2)\ba^2+\bb^2}}.
\end{eqnarray}
Such expression for Berwald's metrics are quite different from which we provided in Section \ref{sq}. However, it is closer to the classical Berwald's metric (\ref{BerwaldF}) in a sense.

The non-uniqueness of $\b$-deformations can be observed in the proof of Lemma \ref{step1}. It it clear that one can choose $\kappa=0$ as the deformation factor when $k_2=0$. However, here we choose an universal deformation factor.

Come back to Berwald's metrics. We can also obtain the expression (\ref{Berwaldepsimple}) by the $\b$-deformations using here, but if we want to do so we need to change the expression of Berwald's metrics. Take
$$\ps=\f{(\sqrt{1+s^2}+s)^2}{\sqrt{1+s^2}},$$
which is a solution of (\ref{ode}) with $k_1=3$, $k_2=0$, and $k_3=-2$. Then by Theorem \ref{main} we can see that $F=\pab$ is locally projectively flat if and only if $\a=(1+\bar b^2)\sqrt{\ba^2-(1+\bar b^2)^{-1}\bb^2}$ and $\b=\sqrt{1+\bar b^2}\bb$. In this case, $F$ is expressed as (\ref{Berwaldepsimple}).

About the non-uniqueness we will go to discuss deeply in Section \ref{non-uniqueness}.

\section{Solutions of Equation (\ref{ode})}\label{6}

In this section we will provide the solves of (\ref{ode}) in integral form. Firstly, it is worth to mention that the related function $f(s)$ are surprisingly similar to the deformation factor (\ref{etab}), but we don't known how to explain such fantastic phenomenon.

On the other hand, recall that for a given $\ab$-metric $F=\pab$, the corresponding function $\ps$ must be positive on some symmetric open interval $(-b_o,b_o)$, so after necessary scaling we can {\it always} assume $\p(0)=1$.
\begin{lemma}\label{solves}
The solutions of equation (\ref{ode}) with the initial conditions
$\phi(0)=1,~\phi'(0)=\epsilon$ are given by
\begin{eqnarray}\label{solution}
\ps=1+\epsilon s+\int^s_0\int^\tau_0\f{k_1+k_2\sigma^2}{1+(k_1+k_3)\sigma^2
+k_2\sigma^4}f(\sigma)\,\ud\sigma\ud\tau.\label{solvesofphi}
\end{eqnarray}
The function $f(s)$ is determined by the coefficients $k_1,k_2,k_3$ and given in the following five case,
\begin{enumerate}
\item When $k_2=0,~k_1+k_3=0$,
$$f(s)=\exp\left\{-\f{k_1s^2}{2}\right\};$$
\item When $k_2=0,~k_1+k_3\neq0$,
$$f(s)=\left\{1+(k_1+k_3)s^2\right\}^{-\f{k_1}{2(k_1+k_3)}};$$
\item When $k_2\neq0,~\Delta_1>0$,
$$f(s)=\f{\left\{\f{\sqrt{\Delta_1}+k_1+k_3}{\sqrt{\Delta_1}-k_1-k_3}\cdot
\f{\sqrt\Delta_1-k_1-k_3-2k_2s^2}{\sqrt\Delta_1+k_1+k_3+2k_2s^2}\right\}
^\f{k_3-k_1}{4\sqrt\Delta_1}}{\sqrt[4]{1+(k_1+k_3)s^2+k_2s^4}};$$
\item When $k_2\neq0,~\Delta_1=0$,
$$f(s)=\f{\sqrt{2}\exp\left\{\f{k_1-k_3}{k_1+k_3}\left[\f{1}{2+(k_1+k_3)s^2}-\f{1}{2}\right]\right\}}
{\sqrt{2+(k_1+k_3)s^2}};$$
\item When $k_2\neq0,~\Delta_1<0$,
$$f(s)=\f{\exp\left\{\f{k_3-k_1}{2\sqrt{-\Delta_1}}
\left(\arctan\f{k_1+k_3+2k_2s^2}{\sqrt{-\Delta_1}}-\arctan\f{k_1+k_3}{\sqrt{-\Delta_1}}\right)\right\}}
{\sqrt[4]{1+(k_1+k_3)s^2+k_2s^4}},$$
\end{enumerate}
where $\Delta_1:=(k_1+k_3)^2-4k_2$.
\end{lemma}
\newpage
\begin{proof}
Set $f(s):=\ps-s\p'(s)$, then
$f'(s)=-s\p''(s)$. Now (\ref{ode}) becomes a equation about $f(s)$ as following
\begin{eqnarray}
\{1+(k_1+k_3)s^2+k_2s^4\}f'(s)=-s(k_1+k_2s^2)f(s),\label{f}
\end{eqnarray}
and the initial conditions turn to be $f(0)=1$. So
$$f(s)=\exp\left\{\int_0^s\f{-t(k_1+k_2t^2)}{1+(k_1+k_3)t^2+k_2t^4}\ud t\right\}.$$
Notice that the similarity between the above integral and (\ref{etab}), one can easily get the analytic expression of $f(s)$ in terms of $\eta$ listing in Theorem \ref{main} by exchanging the role of $k_1$ and $k_3$.

As a result of the equalities of $f'(s)=-s\p''(s)$ and (\ref{f}), $\ps$ satisfies
$$\p''(s)=\f{k_1+k_2s^2}{1+(k_1+k_3)s^2+k_2s^4}f(s),$$
so it can be expressed as (\ref{solvesofphi}) under the given initial conditions.
\end{proof}

In a sense, (\ref{solution}) is the best form for the solutions, because most of them are non-elementary. But if we consider some specific parameters, the expressions can be more explicit. Take $k_1=\f{1}{p}$, $k_2=0$ and $k_3=\f{r-1}{p}$, then (\ref{ode}) becomes
\begin{eqnarray}\label{sp}
\ps-s\p'(s)=(p+rs^2)\p''(s).
\end{eqnarray}
The corresponding $\ab$-metrics are discussed in \cite{CL} and \cite{S}. Notice that $\Delta_1>0$ when $r\neq0$ and $\Delta_1=0$ when $r=0$. Denote $\phi_{r,p}(s)$ the solutions of (\ref{sp}) with the initial
conditions $\phi(0)=1$ and $\phi'(0)=\epsilon$, by Lemma \ref{solves} we have
$$\phi_{r,p}(s)=1+\epsilon s+\f{1}{p}\int_0^s\int_0^\tau\left(1+\frac{r}{p}\sigma^2\right)^{-\f{1}{2r}-1}\ud\sigma\ud\tau.$$

When $r\neq0$, some explicit solutions are listed below, in which $n$ is a positive integer and $\delta=\pm1$:
\begin{eqnarray*}
\phi_{-\frac{1}{2n},\frac{\delta}{2n}}(s)&=&1+\epsilon
s+2n\sum_{m=0}^{n-1}\frac{(-1)^m\delta^{m+1}C^m_{n-1}s^{2m+2}}{(2m+2)(2m+1)},\\
\phi_{\frac{1}{2n},\frac{1}{2n}}(s)&=&\epsilon
s+\frac{(2n-1)!!}{(2n-2)!!}(1+s\arctan s)\\
&&-\sum_{k=1}^{n-1}\frac{(2n-1)!!(2k-2)!!}{(2n-2)!!(2k+1)!!}\frac{1}{(1+s^2)^k},\\
\phi_{\frac{1}{2n},-\frac{1}{2n}}(s)&=&\epsilon
s+\frac{(2n-1)!!}{(2n-2)!!}\left(1+\frac{1}{2}s\ln\frac{1-s}{1+s}\right)\\
&&-\sum_{k=1}^{n-1}\frac{(2n-1)!!(2k-2)!!}{(2n-2)!!(2k+1)!!}\frac{1}{(1-s^2)^k},\\
\phi_{-\frac{1}{2n-1},-\frac{1}{2n-1}}(s)&=&\epsilon
s+\frac{(2n-1)!!}{(2n-2)!!}\left(\sqrt{1+s^2}-s\ln(s+\sqrt{1+s^2})\right)\\
&&-\sum_{k=1}^{n-1}\frac{(2n-1)!!(2k-2)!!}{(2n-2)!!(2k+1)!!}(1+s^2)^{\frac{2k+1}{2}},
\end{eqnarray*}
\begin{eqnarray*}
\phi_{-\frac{1}{2n-1},\frac{1}{2n-1}}(s)&=&\epsilon
s+\frac{(2n-1)!!}{(2n-2)!!}\left(\sqrt{1-s^2}+s\arcsin s\right)\\
&&-\sum_{k=1}^{n-1}\frac{(2n-1)!!(2k-2)!!}{(2n-2)!!(2k+1)!!}(1-s^2)^{\frac{2k+1}{2}},\\
\phi_{\frac{1}{2n-1},\frac{\delta}{2n-1}}(s)&=&\epsilon
s+\frac{(2n-2)!!}{(2n-3)!!}\frac{1+2\delta s^2}{2\sqrt{1+\delta s^2}}\\
&&-\sum_{k=2}^{n-1}\frac{(2n-2)!!(2k-3)!!}{(2n-3)!!(2k)!!}(1+\delta
s^2)^{-\frac{2k-1}{2}}~(n\geq2).
\end{eqnarray*}

When $r=0$, the solutions are given as the following power series,
\begin{eqnarray*}
\phi_{0,p}(s)&=&1+\epsilon
s+\frac{1}{p}\sum_{n=0}^{\infty}\frac{(-1)^ns^{2n+2}}{(2n+2)(2n+1)n!(2p)^n}.
\end{eqnarray*}

According to the above discussion and Theorem \ref{main}, we can obtain the following non-trivial locally projectively flat $\ab$-metrics. The corresponding data $\ba$ and $\bb$ all satisfy the conditions in Theorem \ref{main}, and they will be determined in the last section.
\begin{example}\label{ex1}
Take $k_1=2\sigma$, $k_2=0$ and $k_3=-2\sigma-1$, in order to insure $k_2\neq k_1k_3$, $\sigma$ shouldn't be equal to $0$ or $-\frac{1}{2}$, then
$$F=\left(1-\bar b^2\right)^{-\sigma-1}\sqrt{\left(1-\bar b^2\right)\ba^2+\bb^2}
\phi_{\sigma}\left(\f{\bb}{\sqrt{\left(1-\bar
b^2\right)\ba^2+\bb^2}}\right)$$
are locally projectively flat, where $\p_\sigma(s):=\phi_{-\frac{1}{2\sigma},\frac{1}{2\sigma}}(s)$ and can be expressed as (\ref{exampleonball}). In this case, $\Delta_1>0$.
\end{example}

\begin{example}\label{ex2}
Take $k_2=0,~k_1=-k_3=\pm 2$, then
$$F=e^{\bar b^2}\left\{\ba+\epsilon\bb+2\sum^\infty_{n=0}\f{(-1)^n\bb^{2n+2}}{(2n+2)(2n+1)n!\ba^{2n+1}}\right\}$$
and
$$F=e^{-\bar b^2}\left\{\ba+\epsilon\bb-2\sum^\infty_{n=0}\f{\bb^{2n+2}}{(2n+2)(2n+1)n!\ba^{2n+1}}\right\}$$
are both locally projectively flat. In this case, $\Delta_1=0$.
\end{example}

\begin{example}
Take $k_1=k_3=0,k_2=1$, then
$$F=\left(1+\bar b^4\right)^{-\f{3}{4}}\left\{\sqrt{(1+\bar b^4)\ba^2-\bar
b^2\bb^2}\left[1+\int_0^s
\int_0^\tau\f{\sigma^2}{(1+\sigma^4)^{\f{5}{4}}}{\mathrm
d}\sigma{\mathrm d}\tau\right]+\epsilon\bb\right\}$$
is locally projectively flat, where $s=\f{\bb}{\sqrt{(1+\bar
b^4)\ba^2-\bar b^2\bb^2}}$. In this case, $\Delta_1<0$.
\end{example}

Here the examples are all typical, especially when $\Delta_1\geq0$. Actually, we will see in Section \ref{non-uniqueness} that all the possible locally projectively flat $\ab$-metrics with $\Delta_1\geq0$ have been listed completely in Example \ref{ex1} and Example \ref{ex2}.

\section{Non-uniqueness of expressions for $\ab$-metrics}\label{non-uniqueness}

As we have pointed out in Section \ref{basic}, any given $\ab$-metric has infinity many different expressional forms. Thus, there are not so many $\ab$-metrics that it looks like in Theorem \ref{main}. The aim of this section is to study how many non-trivial locally projectively flat $\ab$-metrics exactly and how to judge two different metrics in forms are essentially the same one or not.

Given an $\ab$-metric $F=\pab$. Let $u$ be a constant number such that $1-ub^2>0$. Then
$\check\a:=\sqrt{\a^2-u\b^2}$ is still a
Riamannian metric. Hence, $F$ can be rewrote in terms of $\check\a$ and $\b$ as
\begin{eqnarray*}
F=\sqrt{\check\a^2+u\b^2}\phi\left(\f{\b}{\sqrt{\check\a^2+u\b^2}}\right)
=\check\a\psi(\f{\b}{\check\a}),
\end{eqnarray*}
where $\psi(s)=\sqrt{1+us^2}\phi(\f{s}{1+us^2})$. Similarly, let $\check\b:=\frac{\b}{v}$, where $v$ is a non-zero constant, then $F$ can be rewrote in terms of $\a$ and $\check\b$ as
$$F=\a\phi\left(v\f{\check\b}{\a}\right)=\a\varphi(\frac{\check\b}{\a}),$$
where $\varphi(s)=\phi(vs)$. %In other words, although two $\ab$-metrics 表面上看起来不一样, with different function, different Riemannian metric and different $1$-form, but it is possible that they essentially make no difference, just we have mentioned in Section \ref{basic}.

Base on the above argument, we introduce two special transformations for the function $\p$:
$$g_u(\ps):=\sqrt{1+us^2}\phi(\f{s}{1+us^2}),\quad h_v(\ps):=\p(vs),$$
where $u$ and $v$ are constants with $v\neq0$. Their compositions are given by
\begin{eqnarray*}
g_{u_1}\circ g_{u_2}=g_{u_1+u_2},\quad h_{v_1}\circ h_{v_2}=h_{v_1v_2},\quad h_v\circ g_u=g_{uv^2}\circ h_v.
\end{eqnarray*}
Hence, these two kinds of transformations generate a transformation group $G$ with the above generation relationships.
\begin{definition}\label{type}
Two $\ab$-metrics $F_1=\a_1\p_1(\f{\b_1}{\a_1})$ and $F_2=\a_2\p_2(\f{\b_2}{\a_2})$ are said to be of \emph{same type} if there is a element $\pi\in G$ such that $\pi(\p_1)=\p_2$. In this case, the functions $\p_1(s)$ and $\p_2(s)$ are said to  be \emph{equivalent}. $G$ is called the \emph{representation group} of $\ab$-metrics.
\end{definition}
For example, all the functions equivalent to $1+s$ will provide Randers type metrics. Conversely, if $F=\pab$ is of Randers type, then $\p(s)$ must be equivalent to $1+s$. Actually, the functions for Randers type metrics, which are given by $\ps=\sqrt{1+us^2}+vs$, can be expressed as $\ps=g_u\circ h_v(1+s)$. Notice that all the functions are always asked to satisfy $\p(0)=1$.

Suppose that a given locally projectively flat $\ab$-metric $F=\pab$ is neither locally Minkowskian nor of Randers type, then $\ps$ must be a solution of (\ref{ode}) according to Z.Shen's result. Due to the non-uniqueness, if we rewrite the metric as $F=\check\a\psi(\f{\check\b}{\check\a})$, then the new function $\psi(s)$, which is equivalent to $\ps$, must be also a solution of (\ref{ode}) with some different parameters. This property is formulated below.
\begin{lemma}\label{change1}
Denote $\psi(s):=g_u(\p)$, where $\ps$ is the solution of (\ref{ode}) with the initial conditions $\p(0)=1$ and $\p'(0)=\epsilon$. Then
$$\left\{1+(k_1'+k_3')s^2+k_2's^4\right\}\psi''(s)=(k_1'+k_2's^2)\left\{\psi(s)-s\psi'(s)\right\},$$
where the constant $k_1'$, $k_2'$ and $k_3'$ are given by
$$k_1'=k_1+u,~k_3'=k_3+u,~k_2'=k_2+(k_1+k_3)u+u^2.$$
Moreover, $\psi(0)=1$ and $\psi'(0)=\epsilon$.
\end{lemma}

\begin{lemma}\label{change2}
Denote $\varphi(s):=h_v(\p)$, where $\ps$ is the solution of (\ref{ode}) with the initial conditions $\p(0)=1$ and $\p'(0)=\epsilon$. Then
$$\left\{1+(k_1''+k_3'')s^2+k_2''s^4\right\}\varphi''(s)=(k_1''+k_2''s^2)\left\{\varphi(s)-s\varphi'(s)\right\},$$
where the constant $k_1''$, $k_2''$ and $k_3''$ are given by
$$k_1''=v^2k_1,~k_3''=v^2k_3,~k_2''=v^4k_2.$$
Moreover, $\varphi(0)=1$ and $\varphi'(0)=v\epsilon$.
\end{lemma}

Before the further discussions, it should be pointed out that when $k_2\neq k_1k_3$, which is our biggest concern, the solutions of the equation (\ref{ode}) with the initial conditions $\p(0)=1$ and $\p'(0)=\epsilon$ are one-to-one correspondence to the quadruple data $(k_1,k_2,k_3,\epsilon)$ since the polynomial $1+(k_1+k_3)s^2+k_2s^4$ is not divisible by $k_1+k_2s^2$. Therefor, what we will do next is to determine the \emph{complete system of invariants} for the solution space
$$\Phi:=\{\ps~|~\ps~\textrm{satisfies Eqn. (\ref{ode}) and}~\p(0)=1\}$$
under the action of the representation group $G$ in terms of the quadruple data.

Define three variables depended on $\ps$:
$$\Delta_1:=(k_1+k_3)^2-4k_2,\quad\Delta_2:=4(k_1k_3-k_2),\quad\Delta_3:=k_1-k_3.$$
It is obviously that $\Delta_1-\Delta_2=\Delta_3^2$, and $\Delta_2=0$ if and only if $F=\pab$ is of Randers type.

If $\psi(s)=g_u(\phi(s))$, then by Lemma \ref{change1} $\Delta'_i=\Delta_i$ for $i=1,2,3$, where $\Delta_i'$ are determined by $k_i'$. Similarly, if $\varphi=h_v(\p(s))$, by Lemma \ref{change2} $\Delta''_i=v^4\Delta_i$ for $i=1,2$ and $\Delta''_3=v^2\Delta_3$. Hence, $\textrm{sgn}(\Delta_i)$ are invariant under the action of $G$.

Due to the sign of $\Delta_1$, the three-parameters equation (\ref{ode}) can be simplified as an one-parameter equation.

\begin{enumerate}
\item $\Delta_1>0$

In this case, $k_2$ must be non-negative.
Let $u_1$ and $u_2$ be the roots of the quadratic equation
$u^2+(k_1+k_3)u+k_2=0$. If $F=\gab$ is a {\it regular} Finsler metric, then $(1-u_1s^2)(1-u_2s^2)>0$ according to Lemma \ref{condition2}, which implicates $1-u_ib^2>0$ for $i=1,2$. So $u=u_i$ are both applicable as transformation factors. Taking anyone of them as the factor then we have $k_2'=k_2+(k_1+k_3)u+u^2=0$. That is to say, when $\Delta_1>0$, by Lemma \ref{change1} the equation (\ref{ode}) can always be reduced as
$$\left\{1+(k_1+k_3)s^2\right\}\p''(s)=k_1\{\ps-s\p'(s)\}.$$
Furthermore, in this case the roots of the equation $u^2+(k_1+k_3)u=0$ are given by $u_1=0$ and $u_2=-(k_1+k_3)$. $u_2$ must be non-zero because
\newpage
\noindent $\Delta_1>0$. Hence, if we take $u=-(k_1+k_3)$ as the transformation factor, then by Lemma \ref{change1} the above equation turns to be
$$\left\{1-(k_1+k_3)s^2\right\}\p''(s)=-k_3\{\ps-s\p'(s)\}.$$
So we can assume the coefficient of the term $s^2$ for the reduced equation is positive. Finally, by Lemma \ref{change2} the equation (\ref{ode}) can be reduced as
\begin{eqnarray}\label{D1p}
\left\{1-s^2\right\}\p''(s)=2\sigma\{\ps-s\p'(s)\},\quad\sigma\neq0,-\frac{1}{2}.
\end{eqnarray}
Here the restriction on $\sigma$ ensures $\Delta_2\neq0$.
\vspace{-5bp}
\item $\Delta_1=0$

Just like the first case above, (\ref{ode}) can be reduced such that $k_2=0$, and at the same time $k_1+k_3=0$ because $\Delta_1=0$, i.e.,
$$\p''(s)=k_1\{\ps-s\p'(s)\},$$
where $k_1\neq0$ since $\Delta_2\neq0$. So by Lemma \ref{change2} the above equation can be reduced finally as
\begin{eqnarray}\label{D1z}
\p''(s)=2\sigma\{\ps-s\p'(s)\},\quad\sigma=\pm1.
\end{eqnarray}
\vspace{-0.7cm}
\item $\Delta_1<0$

In this case, $k_2$ must be positive. There is no real root of the
quadratic equation $u^2+(k_1+k_3)u+k_2=0$. But according to Lemma \ref{condition2}, $1+k_1b^2>0$ when $F=\pab$ is a {\it regular} Finsler metric. So we can take
$u=-k_1$ as the transformation factor and hence (\ref{ode}) can be reduced as
$$\left\{1+k_3s^2+k_2s^4\right\}\p''(s)=k_2s^2\{\p(s)-s\p'(s)\}.$$
So by Lemma \ref{change2} the above equation can be reduced finally as
\begin{eqnarray}\label{D1n}
\left\{1+2\sigma s^2+s^4\right\}\p''(s)=s^2\{\ps-s\p'(s)\},\quad|\sigma|<1.
\end{eqnarray}
The restriction on $\sigma$ ensures $\Delta_1<0$.
\end{enumerate}

Suppose that $F=\pab$ is a non-trivial projectively flat $\ab$-metric on $U\subseteq\RR^n$ with $n\geq3$, where $\ps$ satisfies (\ref{ode}) with the parameters $k_1$, $k_2$ and $k_3$ and the initial condition $\p(0)=1$ and $\p'(0)=\epsilon$. Define a couple of variables $(p,q)_\p$ determined by the quadruple data $(k_1,k_2,k_3,\epsilon)$ as following:
\vspace{-2bp}
$$(p,q)_\p:=\left(\f{\sqrt{\Delta_2}}{\Delta_3},\f{\epsilon^4}{\Delta_2}\right).$$
Some special case are defined below:
\vspace{-5bp}
\begin{itemize}
\item $p:=0$ when $\Delta_2=0$;
\vspace{-4bp}
\item $p:=\infty$ when $\Delta_2>0$ and $\Delta_3=0$;
\vspace{-4bp}
\item $p:=i\infty$ when $\Delta_2<0$ and $\Delta_3=0$;
\vspace{-4bp}
\item $q:=0$ when $\epsilon=0$;
\vspace{-4bp}
\item $q:=\infty$ when $\epsilon\neq0$ and $\Delta_2=0$.
\end{itemize}
It is easy to see that $(0,0)_\phi$ and $(0,\infty)_\phi$ correspond to Riemannian metrics and non-Riemann Randers metrics respectively.

By Lemma \ref{change1} and \ref{change2} we can see that these variables are invariable under
the action of $G$, i.e., $(p,q)_{\pi(\p)}=(p,q)_{\p}$ for any $\pi\in G$. The values of $p$ (for the
reduced equations), which is the key invariant, are listed in Table \ref{table}.
\newpage

\begin{table}[htbp]
  \centering
  \caption{}\label{table}
  \begin{tabular}{|c|c|c|c|c|c|}
  \hline
  % after \\: \hline or \cline{col1-col2} \cline{col3-col4} ...
                  &$k_1$    &$k_2$& $k_3$      & $p$ & range of $p$ \\
  \hline
  Eqn. (\ref{D1p})&$2\sigma$& $0$ &$-2\sigma-1$&$\frac{2\sqrt{-2\sigma(2\sigma+1)}}{4\sigma+1}$&$\RR\backslash\{0\}\cup\{\infty\}\cup i(-1,1)$\\
  \hline
  Eqn. (\ref{D1z})&$2\sigma$& $0$ &$-2\sigma$  &$\frac{\sqrt{-\sigma^2}}{\sigma}$&$\{\pm i\}$ \\
  \hline
  Eqn. (\ref{D1n})&$0$      & $1$ &$2\sigma$   &$-\frac{i}{\sigma}$&$i(-\infty,-1)\cup i(1,+\infty)\cup \{i\infty\}$ \\
  \hline
\end{tabular}
\end{table}

there is no intersection between the values of $p$ for different class of reduced equations, Moreover, for the same class of reduced equations, different parameters will correspond to different values of $p$. So the reduced equations are one-to-one correspondence to the variable $p$, the range of which is $\RR\backslash\{0\}\cup\{\infty\}\cup i\RR\cup\{i\infty\}$. Combining with Lemma \ref{change2} we have
\begin{proposition}\label{IV}
When dimension $n\geq3$, two non-trivial locally projectively flat $\ab$-metrics $F_i=\a_i\p_i(\frac{\b_i}{\a_i})~(i=1,2)$ are of the same type if and only if $(p,q)_{\p_1}=(p,q)_{\p_2}$.
\end{proposition}

The about result indicates that $(p,q)_\p$ are the complete system of invariants we need. As a result, we immediately have the following
\begin{theorem}[Classification]\label{projab}
Let $F=\pab$ be an $\ab$-metric on a $n$-dimensional manifold $M$ with $n\geq3$. Then $F$ is locally projectively flat if and only if $F$ lies in one of the following cases:
\begin{enumerate}
\item $\a$ is projectively flat and $\b$ is parallel with respect to $\a$. In this case, $F$ is either a Riemannian metric with constant sectional curvature or a locally Minkowski metric;
\item $F$ is a locally projectively flat Randers metric;
\item On the open subset $\mathcal U$ of $M$ where $\ud b\neq0$ everywhere or $\ud b\equiv0$, $F$ can be reexpressed (if necessary) still as the form $F=\pab$ such that one of the following holds
\begin{enumerate}
\item $\ps$ satisfies Eqn. (\ref{D1p}), $\a$ and $\b$ are determined by
\begin{eqnarray*}
\a&=&(1-\bar
b^2)^{-\sigma-\frac{1}{2}}\sqrt{\ba^2+(1-\bar
b^2)^{-1}\bb^2},\\
\b&=&(1-\bar b^2)^{-\sigma-1}\bb;
\end{eqnarray*}
\item $\ps$ satisfies Eqn. (\ref{D1z}), $\a$ and $\b$ are determined by
\begin{eqnarray*}
\a=e^{\sigma\bar b^2}\ba,\qquad\b=e^{\sigma\bar b^2}\bb;
\end{eqnarray*}
\item $\ps$ satisfies Eqn. (\ref{D1n}), $\a$ and $\b$ are determined by
\begin{eqnarray*}
\a&=&\frac{\exp\left(-\frac{\sigma}{2\sqrt{1-\sigma^2}}\arctan\frac{\sigma+\bar
b^2}{\sqrt{1-\sigma^2}}\right)}
{(1+2\sigma\bar b^2+\bar b^4)^\frac{1}{4}}\sqrt{\ba^2-\frac{(2\sigma+\bar b^2)}{1+2\sigma\bar b^2+\bar b^4}\bb^2},\\
%\end{eqnarray*}
%\begin{eqnarray*}
\b&=&\frac{\exp\left(-\frac{\sigma}{2\sqrt{1-\sigma^2}}\arctan\frac{\sigma+\bar
b^2}{\sqrt{1-\sigma^2}}\right)} {\left(1+2\sigma\bar b^2+\bar
b^4\right)^\frac{3}{4}}\bb.
\end{eqnarray*}
\end{enumerate}
In the three case above, $\ba$ is a locally projectively flat Riemannian metric, $\bb$ is a closed 1-form which is conformal with respect to $\ba$, $\bar b:=\|\bb\|_{\ba}$.
\end{enumerate}
\end{theorem}
\newpage

\section{Reversible projectively flat $\ab$-metrics}\label{s9}
By Lemma \ref{solves} we known that the solutions of (\ref{ode}) have a distinctive feature. Because $f(s)$ is an even function, except for the term $\epsilon s$, the rest part of $\ps$ is even too. This observation leads to the following result.
\begin{theorem}\label{pr}
Let $F=\pab$ be a non-trivial projectively flat $\ab$-metric on an open subset $U\subseteq\RR^n$ with $n\geq3$. Then exist a closed $1$-form $\theta=\theta_i(x)y^i$ on $U$ such that $F+\theta$ is a reversible projectively flat $\ab$-metric.
\end{theorem}
\begin{proof}[Proof of Theorem \ref{pr}]
Recall that for a Randers metric $F=\a+\b$, $F$ is projectively flat if and only if $\a$ is projectively flat and $\b$ is closed, so it is obviously true for Randers metrics.

If $F=\pab$ isn't of Randers type. then $\ps$ must satisfy (\ref{ode}), and $\tilde\phi(s):=\ps-\p'(0)s$, which is still a solution of (\ref{ode}), is an even function. Since $F$ is a Finsler metric, $\ps$ satisfies $\phi(s)>0$ and (\ref{conditionofphi}) when $|s|\leq b<b_o$. It is easy to verify that $\tilde\phi(s)-s\tilde\phi'(s)=\ps-s\p'(s)$ and $\tilde\phi''(s)=\p''(s)$, so $\tilde\p(s)$ satisfies (\ref{conditionofphi}) too. On the other hand, $\tilde\p(s)$ must be positive when $|s|<b_o$ since
$$\tilde\p(s)\geq\min\{\tilde\p(s)-\p'(0)s,\tilde\p(s)+\p'(0)s\}=\min\{\ps,\p(-s)\}.$$
So $\tilde F=\a\tilde\p(\f{\b}{\a})$, namely $F+\theta$ where $\theta=-\p'(0)\b$, is a reversible projectively flat $\ab$-metric.
\end{proof}

Recall that Rapcs\'{a}k's result tells us that if $F$ is a projectively flat Finsler metric, then for any closed $1$-form $\theta$, $F+\theta$ is projectively flat too as long as it is still a Finsler metric. Hence, Theorem \ref{pr} implies in a sense that there isn't any non-trivial irreversible projectively flat $\ab$-metrics. But it doesn't mean that the irreversible projectively flat $\ab$-metrics are utterly useless. For instance, Berwald's metric (\ref{BerwaldF}), which has a great property to be of constant flag curvature, is irreversible. But for the corresponding reversible metric $F=\a+\frac{\b^2}{\a}$, it couldn't be of constant flag curvature except for the trivial case.

At the end of this section, let's talk about the quantity of non-trivial reversible projectively flat $\ab$-metrics.

It is obviously that $F=\pab$ is a non-trivial reversible projectively flat $\ab$-metric if and only if $q_\p=\frac{\epsilon^4}{\Delta_2}=0$. Therefor, $p_\p=\frac{\sqrt{\Delta_2}}{\Delta_3}$ is the unique \emph{invariant} to distinguish different type of reversible metrics, and its range is $\RR\cup\{\infty\}\cup i\RR\cup\{i\infty\}$. Notice that $p=0$ here represents Riemannian metrics, which are also one type of non-trivial reversible metrics.
\begin{theorem}
There are a total of one-parameter types of non-trivial reversible locally projectively flat $\ab$-metrics when dimension $n\geq3$.
\end{theorem}

We can use a equivalent variable
$$\left(\f{2\sqrt{|\Delta_2|}\Delta_3}{|\Delta_2|+\Delta_3^2},
\f{2\Delta_3}{|\Delta_2|+\Delta_3^2}\right)_\p$$
to replace $p$ as the invariant. The points become two circles $x^2+(y\pm1)^2=1$ in plane, and every different point corresponds to a different type of metrics. In particular, the origin represents Riemannian metrics. More information are shown in Figure \ref{tu}.
\begin{figure}[htbp]
  \includegraphics[width=0.80\textwidth]{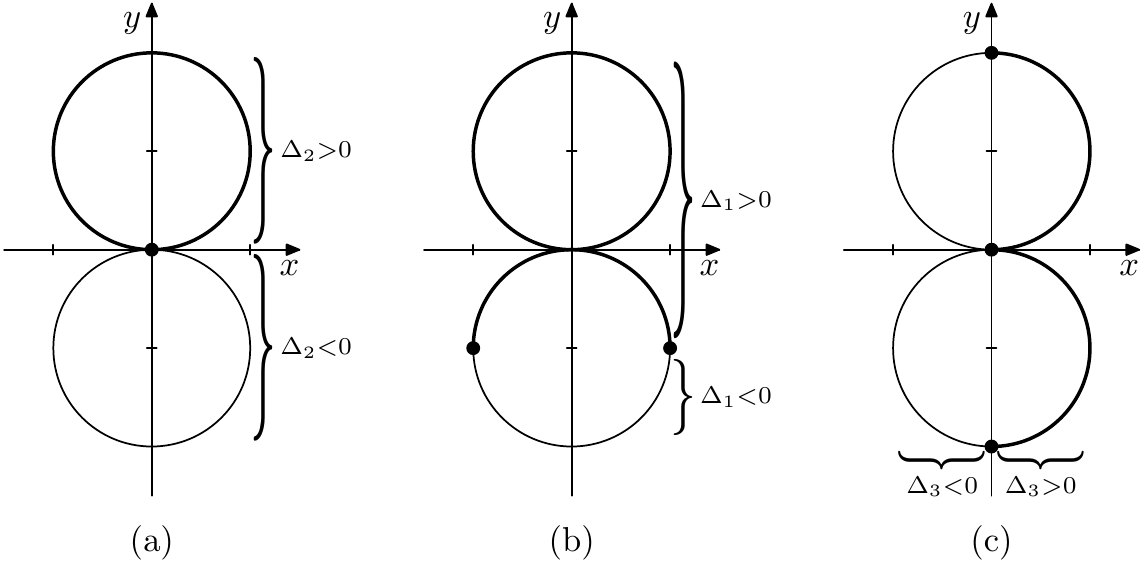}\\
  \caption{}\label{tu}
\end{figure}

\section{Complete list of allowable $1$-forms}\label{7}
One can see that the closed and conformal $1$-forms with respect to the constant curvature Riemannian metrics play an important role here. Our aim in this section is to derive their explicit expression.

Let $h$ be the Riemannian metric with constant sectional curvature $\mu$. $h$ can be expressed as
\begin{eqnarray}\label{h}
h=\f{\sqrt{(1+\mu|x|^2)|y|^2-\mu\langle x,y\rangle^2)}}{1+\mu|x|^2}.
\end{eqnarray}
Under the coordinate system using here, all the geodesics of $h$ are straight lines.

According to \cite{SX}, when dimension $n\geq3$, the conformal vector fields with respect to $h$ are given by
\begin{eqnarray}\label{W3}
\quad W=\left(\lambda\sqrt{1+\mu|x|^2}+\langle a,x\rangle\right)x-\f{|x|^2a}{\sqrt{1+\mu|x|^2}+1}+qx+b+\mu\langle b,x\rangle x
\end{eqnarray}
where $\lambda$ is a constant, $q=(q_{ij})\in so(n)$ is an anti-symmetric
matrix and $a,b\in\RR^n$ are constant vectors. When dimension $n=2$, they are given by
\begin{eqnarray}
\quad W=\left\{(1+\mu x_1^2)P+\mu x_1x_2Q\right\}\pp{}{x^1}+\left\{(1+\mu x_2^2)Q+\mu x_1x_2P\right\}\pp{}{x^2}\label{W2}
\end{eqnarray}
where $x:=(x_1,x_2)$ denote the coordinate of points, the functions $P(x_1,x_2)$ and $Q(x_1,x_2)$ are determined by two equations as following:
\begin{eqnarray}
&\displaystyle \pp{P}{x_2}+\pp{Q}{x_1}=-\mu x_1x_2\left\{\f{1}{1+\mu x_2^2}\pp{P}{x_1}+\f{1}{1+\mu x_1^2}\pp{Q}{x_2}\right\},\label{PQ1}\\
%\end{eqnarray}
%\begin{eqnarray}
&\displaystyle \f{1}{1+\mu x_2^2}\pp{P}{x_1}-\f{1}{1+\mu x_1^2}\pp{Q}{x_2}=0.\label{PQ2}
\end{eqnarray}

By the above facts, we will show that the vector fields $W$ are simple and can be expressed in an unified form if adding the condition $\ud W^\flat=0$.

\begin{lemma}\label{cc}
Let $h$ be a locally projectively flat Riemannian metric and $W$ be
a conformal vector field with respect to $h$. If the dual 1-form $W^\flat$ of $W$ with respect of $h$ is closed, then there is a local coordinate system
in which $h$ is given by (\ref{h}) and
\begin{eqnarray*}
W=\sqrt{1+\mu|x|^2}(\lambda x+a),
\end{eqnarray*}
where $\mu,~\lambda$ are constants and $a\in\RR^n$ is a constant vector. In this case,
\begin{eqnarray*}
W^\flat=\f{\lambda\langle x,y\rangle+(1+\mu|x|^2)\langle
a,y\rangle-\mu\langle a,x\rangle\langle
x,y\rangle}{(1+\mu|x|^2)^\f{3}{2}}.
\end{eqnarray*}
%\begin{eqnarray}
%\ba&=&\f{\sqrt{(1+\mu|x|^2)|y|^2-\mu\langle x,y\rangle^2)}}{1+\mu|x|^2},\label{balpha}\\
%\bb&=&
%\end{eqnarray}
\end{lemma}
\begin{proof}
\begin{enumerate}
\item When $n\geq3$,
\begin{enumerate}
\item If $\mu=0$, by (\ref{W3}) we have $W=\lambda x+\langle a,x\rangle x-\f{|x|^2a}{2}+qx+b$. A direct computation shows that
%$$\pp{W_i}{x^j}=\lambda\delta_{ij}+a^ix^j-a^jx^i+\langle a,x\rangle\delta_{ij}+Q_{ij},$$
$$\pp{W_i}{x^j}-\pp{W_j}{x^i}=2\left(a^ix^j-a^jx^i+q_{ij}\right).$$
So $\ud W^\flat=0$ if and only if $a=0$ and $q=0$.
\item If $\mu\neq0$, by (\ref{W3}) $W$ can be reexpressed in a new form as
$$W=\sqrt{1+\mu|x|^2}(\lambda x+a)+qx+b+\mu\langle b,x\rangle x.$$
Note that $a$ and $b$ here are different from that of (\ref{W3}). A direct computation shows that
%$$W_i=\frac{\lambda x^i}{(1+\mu|x|^2)^\frac{3}{2}}+\frac{a^i}{\sqrt{1+\mu|x|^2}}
%-\frac{\mu\langle a,x\rangle %x^i}{(1+\mu|x|^2)^\frac{3}{2}}+\frac{Q_{ik}x^k}{1+\mu|x|^2}+\frac{b^i}{1+\mu|x|^2},$$
%\begin{eqnarray*}
%\pp{W_i}{x^j}&=&\frac{\lambda\delta_{ij}}{(1+\mu|x|^2)^\frac{3}{2}}-\frac{3\mu\lambda
%x^ix^j}{(1+\mu|x|^2)^\frac{5}{2}}-\frac{\mu(a^ix^j+a^jx^i)}{(1+\mu|x|^2)^\frac{3}{2}}-\frac{\mu\langle
%a,x\rangle\delta_{ij}}{(1+\mu|x|^2)^\frac{3}{2}}\\
%&&+\frac{3\mu^2\langle a,x\rangle
%x^ix^j}{(1+\mu|x|^2)^\frac{5}{2}}-\frac{2\mu Q_{ik}x^kx^j}
%{(1+\mu|x|^2)^2}+\frac{Q_{ij}}{1+\mu|x|^2}-\frac{2\mu b^ix^j}{(1+\mu|x|^2)^2},
%\end{eqnarray*}
\begin{eqnarray*}
\pp{W_i}{x^j}-\pp{W_j}{x^i}&=&\frac{2}{(1+\mu|x|^2)^2}\Big\{(1+\mu|x|^2)q_{ij}\\
&&+\mu(q_{jk}x^i-q_{ik}x^j)x^k+\mu(b^jx^i-b^ix^j)\Big\}.
\end{eqnarray*}
Thus, it is easy to see that $\ud W^\flat=0$ if and only if $b=0$ and $q=0$.
\end{enumerate}
\item When $n=2$, by (\ref{W2}) we can see that the corresponding formula of $W^\flat$ is $$W^\flat=\frac{P\ud x_1+Q\ud x_2}{1+\mu(x_1^2+x_2^2)}.$$
    In this case, $\ud W^\flat=0$ is equivalent to the following equality
\begin{eqnarray}
\pp{Q}{x_1}-\pp{P}{x_2}=\frac{2\mu(x_1Q-x_2P)}{1+\mu(x_1^2+x_2^2)}.\label{PQ3}
\end{eqnarray}
By (\ref{PQ1}) and (\ref{PQ3}) we get
\begin{eqnarray}
\pp{P}{x_2}=-\frac{\mu x_1x_2}{1+\mu
x_1^2}\pp{Q}{x_2}-\frac{\mu(x_1Q-x_2P)}{1+\mu(x_1^2+x_2^2)},\label{eqn:36}\\
%\end{eqnarray}
%\begin{eqnarray}
\pp{Q}{x_1}=-\frac{\mu x_1x_2}{1+\mu
x_2^2}\pp{P}{x_1}+\frac{\mu(x_1Q-x_2P)}{1+\mu(x_1^2+x_2^2)}.\label{eqn:37}
\end{eqnarray}
Set
\begin{eqnarray*}
\tilde P=\frac{(1+\mu x_1^2)P+\mu x_1x_2 Q}{\sqrt{1+\mu(x_1^2+x_2^2)}},\quad
\tilde Q=\frac{(1+\mu x_2^2)Q+\mu x_1x_2 P}{\sqrt{1+\mu(x_1^2+x_2^2)}}.
\end{eqnarray*}
\newpage
Direct computations yields
\begin{eqnarray*}
\pp{\tilde
P}{x_1}&=&\frac{1}{\left[1+\mu(x_1^2+x_2^2)\right]^\frac{3}{2}}\Bigg\{\left[1+\mu(x_1^2+x_2^2)\right]\left[(1+\mu
x_1^2)\pp{P}{x_1}+\mu x_1x_2\pp{Q}{x_1}\right]\\
&&+\mu\left[1+\mu(x_1^2+x_2^2)\right](x_1P+x_2Q)+\mu^2x_1x_2(x_2P-x_1Q)\Bigg\},\\
%\end{eqnarray*}
%\begin{eqnarray*}
\pp{\tilde Q}{x_2}&=&\frac{1}{\left[1+\mu(x_1^2+x_2^2)\right]^\frac{3}{2}}\Bigg\{\left[1+\mu(x_1^2+x_2^2)\right]\left[(1+\mu
x_2^2)\pp{Q}{x_2}+\mu x_1x_2\pp{P}{x_2}\right]\\
&&+\mu\left[1+\mu(x_1^2+x_2^2)\right](x_1P+x_2Q)-\mu^2x_1x_2(x_2P-x_1Q)\Bigg\},\\
%\end{eqnarray*}
%\begin{eqnarray*}
\pp{\tilde P}{x_2}%&=&\frac{1}{1+\mu(x_1^2+x_2^2)}\Bigg\{\left[(1+\mu
%x_1^2)\pp{P}{x_2}+\mu x_1Q+\mu
%x_1x_2\pp{Q}{x_2}\right]\sqrt{1+\mu(x_1^2+x_2^2)}\\
%&&-\left[(1+\mu x_1^2)P+\mu x_1x_2Q\right]\frac{\mu
%x_2}{\sqrt{1+\mu(x_1^2+x_2^2)}}\Bigg\}\\
&=&\frac{1}{\left[1+\mu(x_1^2+x_2^2)\right]^\frac{3}{2}}\Bigg\{\left[1+\mu(x_1^2+x_2^2)\right]\left[(1+\mu
x_1^2)\pp{P}{x_2}+\mu x_1x_2\pp{Q}{x_2}\right]\\
&&+\mu(1+\mu x_1^2)(x_1Q-x_2P)\Bigg\},\\
%\end{eqnarray*}
%\begin{eqnarray*}
\pp{\tilde
Q}{x_1}&=&\frac{1}{\left[1+\mu(x_1^2+x_2^2)\right]^\frac{3}{2}}\Bigg\{\left[1+\mu(x_1^2+x_2^2)\right]\left[(1+\mu
x_2^2)\pp{Q}{x_1}+\mu x_1x_2\pp{P}{x_1}\right]\\
&&-\mu(1+\mu x_2^2)(x_1Q-x_2P)\Bigg\}.
\end{eqnarray*}
Then (\ref{PQ2}), (\ref{PQ3}), (\ref{eqn:36}), (\ref{eqn:37}) and the above equalities imply
\begin{eqnarray}
\pp{\tilde P}{x_1}=\pp{\tilde Q}{x_2},\qquad\pp{\tilde P}{x_2}=\pp{\tilde Q}{x_1}=0.
\end{eqnarray}
\end{enumerate}
So $\tilde P_1=\lambda x_1+a_1$ and $\tilde P_2=\lambda x_2+a_2$ for some constant numbers $\lambda$, $a_1$ and $a_2$, which implies that $W=\lambda x+a$ with $a=(a_1,a_2)$.
\end{proof}

\noindent Changtao Yu\\
School of Mathematical Sciences, South China Normal
University, Guangzhou, 510631, P. R. China\\
aizhenli@gmail.com

\iffalse
\textcolor[rgb]{0.00,0.00,1.00}{
\begin{itemize}
\item 如何提二维? 包括几个主要结果, 定理 7.4等
\item 是否给出~$\Delta_1<0$~时解的幂级数形式?
\end{itemize}}
\fi
\end{document}